\documentclass[11pt]{amsart}
\input xy
\xyoption{all}
\usepackage{amsmath,amsthm, amscd, amssymb, amsfonts}
\usepackage[usenames, dvipsnames]{xcolor}
\newtheorem{thm}{Theorem}[section]
 \newtheorem{cor}[thm]{Corollary}
 \newtheorem{lem}[thm]{Lemma}
 
 \newtheorem{defn}[thm]{Definition}
\newtheorem{rem}[thm]{Remark}
 \newtheorem{ex}[thm]{Example}
\title{Examples of (non-)Braided tensor categories}
\author{Mejia Casta\~no Adriana}
\address{
{\bf Mejia Casta\~no Adriana} \\
Departamento de Matem\'aticas y Estad\'istica \\
Facultad de Ciencias B\'asicas \\
Universidad del Norte\\
Barranquilla, Colombia \\
e-mail: sighana25@gmail.com, mejiala@uninorte.edu.co\\
\mbox{} }

\begin{document}

\begin{abstract}
 Are introduced six examples of non-braidable tensor categories which are extensions of the category $Comod(H)$, for $H$ a supergroup algebra; and two examples of braided categories where the only possible braiding is the trivial braiding.
 \\[+2mm]
 \subjclassname{ 18D10}
 \end{abstract}

\maketitle

\section{Introduction }

Braided categories were introduced by Joyal and Street \cite{JS}. They are related to knot invariants, topology and quantum groups, since they can express symmetries. Some examples of braided categories are:
\begin{itemize}
    \item graded modules over a commutative ring,
    \item (co)modules over a (co)quasi-triangular Hopf algebra,
    \item the Braid category, \cite[Section 2.2]{JS},
    \item the center of a tensor category.
\end{itemize} 
In the last example, we begin with a tensor category and construct a braided one. In a general scenario, a natural question is it is possible to construct braidings starting with tensor categories. In particular,  if $G$ is a finite group, can a $G$-extension of a tensor category be braided? In this work we show that this can be done in very few cases. Then, an extension of a braided category is not necessarily braided, so it is really complicated to extend that property.

However, constructing examples of non-braided categories is also  important. A big family of these come from the category of (co)modules of a Hopf algebra without a (co)quasi-triangular structure, see \cite[T 10.4.2]{M}. Masuoka in \cite{Ma1} and \cite{Ma2} constructs explicit examples of non- Quasi-triangular or non-CoQuasi-triangular Hopf algebras. In particular these Hopf algebras can not be obtained from any group algebra by twist (or cocycle) deformation. Other examples were constructed in \cite{HH}.

\medbreak

In the literature there are a few explicit examples of tensor categories, for this reason we construct in \cite{MM} eight tensor categories, following the description introduced in \cite{G} of Crossed Products. These categories extend the module category over certain quantum groups, called \emph{supergroup algebras}. In a few words, a \emph{crossed product tensor category} is, as Abelian category, the direct sum of copies of a fixed tensor category, and the tensor product comes from certain data. Then founding all possible data, we explicitly construct tensor categories.

In the same work \cite{G}, the author also describes all posible braidings over a crossed product. Following this,  three  conditions were introduced to decide if a $G$-crossed product is braidable:
\begin{enumerate}
\item the base category has to be braided,
    \item $G$ has to be Abelian, and the biGalois objects associated to each crossed product have to be trivial,
    \item the $3$-cocycle associated to each crossed product over an specific supergroup algebra has to be trivial, if $G$ is the cyclic group of order $2$.
\end{enumerate}
The goal in the present paper is to obtain all possible braidings over the categories introduced in \cite{MM}. With this, only two categories of the eight found in \cite{MM} are braided with the trivial braiding only, and the other 6 are not braidable.

\medbreak

In  \cite[Thm 6.3]{MM}, using the Frobenius-Perron dimension, we proved that these eight categories are the module category of a quasi-Hopf algebra. Although we do not know how to explicitly compute these algebras, as a Corollary of this work, we know that six of these algebras are non-Quasi-triangular and two are Quasi-triangular only. In particular, we are obtaining information about certain quasi-Hopf algebras without knowing them explicitly; showing how useful it is to work in the category world. In a near future, when we can explicitly describe  these quasi-Hopf algebras, we will already know how their Quasi-triangular structures are.

\section{Preliminaries and notation}

Throughout this paper we shall work over an algebraically closed field $\Bbbk$ of characteristic zero. For basic knowledge of Hopf algebras see \cite{M}. 
Let $H$ be a finite-dimensional Hopf algebra and $A$ be a left $H$-comodule. 
Then $A$ is also a right $H$-comodule with right coaction 
$a\mapsto a_0 \otimes  S(a_{-1})$,
see \cite[Prop 2.2.1(iii)]{AG}. 
A \emph{left $H$-Galois extension of} $A^{co(H)}$ is a left $H$-comodule algebra $(A,\rho)$ such that $A\otimes_{A^{co(H)}} A\to H\otimes A$, $a\otimes b\mapsto (1\otimes a)\rho(b)$ is bijective. Similarly, we define right $H$-Galois extension.

Consider $L$ another finite-dimensional Hopf algebra. 
An \emph{(H, L)-biGalois object} \cite{S} is an algebra $A$ that is a left $H$-Galois extension and a right $L$-Galois extension of the base field $\Bbbk$ such that the two comodule structures make it an $(H, L)$-bicomodule. Two biGalois objects are \emph{isomorphic} if there exists a bijective bicomodule morphism that is also an algebra map.
For $A$ an $(H,L)$-biGalois object,  define the tensor functor
 \begin{equation*}\label{monoidal-eq-hopf} \mathcal F _A: \operatorname{Comod}(L)\to  \operatorname{Comod}(H), \quad \mathcal F _A= A\Box_L -.
 \end{equation*}

By \cite{S}, every tensor functor between comodule categories is one of these, and $\mathcal F _A\simeq \mathcal F _B$ as tensor functors if and only if $A\simeq B$ as biGalois objects.

If $A=H$, then every natural monoidal equivalence $\beta:\mathcal F_H\to\mathcal F_H$ is given by $$f\otimes  \operatorname{id}_X:H\Box_H X\to H\Box_H X, \quad (X,\rho_X)\in \operatorname{Comod}(H), $$
where $f:H\to H$ is a bicomodule algebra isomorphism.

\begin{lem}\label{Sch}
Every natural monoidal equivalence $\operatorname{id}_{\operatorname{Comod}(H)}\to\operatorname{id}_{\operatorname{Comod}(H)}$ is given by $(\varepsilon f\otimes  \operatorname{id}_X)\rho_X.$
\end{lem}
\begin{proof}
For $X\in\operatorname{Comod}(H)$, the coaction induces an isomorphism $X\simeq H\Box_H X$ with inverse induced by $\varepsilon$, the counit. Then $\operatorname{id}_{\operatorname{Comod}(H)}\simeq \mathcal F _H$ as tensor functors. Since all natural monoidal autoequivalences of $\mathcal F _H$ are given by $f\otimes\operatorname{id}_X$ then all natural monoidal autoequivalences of $\operatorname{id}_{\operatorname{Comod}(H)}$ are given by $(\varepsilon f\otimes  \operatorname{id}_X)\rho_X.$ 
\end{proof}

\begin{defn}\cite[Defn 10.1.5]{M} $(H,R)$ is a \emph{Quasi-triangular} (or QT) Hopf algebra if $H$ is a Hopf algebra and there exist $R\in H\otimes  H$, called the \emph{$R$-matrix}, invertible such that $$(\Delta\otimes\operatorname{id})R=R^{13}R^{23},\quad (\operatorname{id}\otimes\Delta)R=R^{13}R^{12},\quad \Delta^{op}(h)=R\Delta(h)R^{-1}, h\in H.$$

Dualizing we can define, $(H,r)$ is a \emph{CoQuasi-triangular} (or CQT) Hopf algebra if $H$ is a Hopf algebra and $r: H\otimes  H\to\Bbbk$, called the \emph{$r$-form}, is a linear functional which is invertible with respect to the convolution multiplication and satisfies for arbitrary $a, b, c\in H$
$$ r(c \otimes  ab) = r(c_1\otimes  b)r(c_2 \otimes  a),\quad r(ab \otimes  c) = r(a \otimes  c_1)r(b \otimes  c_2),$$ $$r(a_1 \otimes  b_1)a_2b_2 = r(a_2 \otimes  b_2)b_1a_1.$$ \end{defn}

\begin{rem} Drinfeld defined a \emph{quantum group} as a non-commutative, non-cocommutative Hopf algebra. Examples of these are the QT Hopf algebras. The importance of quantum groups lies in they allow to construct solutions for the quantum Yang-Baxter equation in statistical mechanics (the R-matrix is a solution of the Yang-Baxter equation).  
An example of quantum group are the supergroup algebras.

A \emph{supergroup algebra} is a supercocommutative Hopf superalgebra of the form $\Bbbk[G] \ltimes\wedge V$, where $G$ is a finite group and $V$ is a finite-dimensional $\Bbbk$-module of $G$. They appear and have an interesting role in the classification of triangular algebras, see \cite[Thm 4.3]{EG}. 
\end{rem}
  
\begin{ex}\label{super} Consider $H=\Bbbk C_2\ltimes\Bbbk V$, for $V$ a 2-dimensional vector space and $C_2$ the 2-cyclic group generated by $u$ with $u\cdot v=-v$ for $v\in V$. 
As an algebra, it is generated by elements $v\in V, g \in C_2$ subject to
relations
$vw + wv = 0; gv = (g \cdot v)g$ for all $v,w \in V, g \in C_2$.
The coproduct and antipode are determined by $$\Delta(v) = v\otimes  1+ u\otimes  v; \Delta(g) = g\otimes  g; S(v) = -uv; S(g) = g^{-1}, \quad v \in V, g \in C_2.$$

Taking $R=\frac{1}{2}(1\otimes  1+1\otimes  u+\otimes  1-u\otimes  u)$,  $(H,R)$ is a QT-Hopf algebra. We can construct a CoQuasi-triangular structure taking $r=R^*$ since $H$ is auto-dual. Then $(H,R^*)$ is a CQT-Hopf algebra. 
\end{ex}

\begin{defn} A \emph{finite tensor category} is a locally finite, $\mathbf{k}$-linear, rigid, monoidal Abelian category $\mathcal{D}$ with $\mathrm{End}_\mathcal{D} \left( \mathbf{1} \right) \cong \mathbf{k}$.
Given a finite group $\Gamma$, a (faithful) \emph{$\Gamma$-grading} on a  finite tensor category $\mathcal D$ is a decomposition 
$\mathcal D=\oplus_{g\in \Gamma} \mathcal D_g$, where $\mathcal D_g$ are full Abelian subcategories of $\mathcal D$ such that
\begin{itemize}
\item[$\bullet$]  $\mathcal D_g\neq 0$;
 \item[$\bullet$] $\otimes:\mathcal D_g\times \mathcal D_h\to \mathcal D_{gh}$ for all
$g, h\in \Gamma.$ 
\end{itemize}
 We have that $\mathcal C:=\mathcal D_e$ is a tensor subcategory of $\mathcal D$. The tensor category $\mathcal D$ is call a $\Gamma$-\emph{extension} of $\mathcal C$. Denote by $[V,g]$ the homogeneous elements in $\mathcal D$, for $V\in \mathcal D_g$, $g\in\Gamma$.

A \emph{braided tensor category} is a tensor category $\mathcal C$ with  natural isomorphisms $c_{X,Y}:X\otimes  Y\to Y\otimes  X$ such that
\begin{equation}\label{bra1}
    \alpha_{V,W,U}c_{U,V\otimes  W}\alpha_{U,V,W}=(\operatorname{id}\otimes  c_{U,W})\alpha_{V,U,W}(c_{U,V}\otimes\operatorname{id}),\end{equation}
\begin{equation}\label{bra2}
    \alpha_{W,U,V}^{-1}c_{U\otimes  V, W}\alpha_{U,V,W}^{-1}=(c_{U,W}\otimes\operatorname{id})\alpha_{U,W,V}^{-1}(\operatorname{id}\otimes  c_{V,W}),
\end{equation}
\end{defn} 

If $(H,r)$ is a CQT-Hopf algebra then $\operatorname{Comod}(H)$ is a braided tensor category where the braiding is given by $c_{V\otimes  W}(x\otimes  y)=r(y_{-1}\otimes  x_{-1})y_0\otimes  x_0$, for all $V,W\in\operatorname{Comod}(H)$.

\medbreak

The following Lemma gives us the first condition to know when an extension can be braided.

\begin{thm} Let $\mathcal D=\oplus_{g\in G}\mathcal D_g$ be a $\Gamma$-extension of $\mathcal C$. If $\mathcal D$ is a braided tensor category then $\mathcal C$ is a braided tensor category.
\end{thm}
\begin{proof} Let $c$ be the braiding of $\mathcal D$, then $c_{[V,e],[W,e]}:[V\otimes  W,e]\to[W\otimes  V,e]$ and  $c_{[V,e],[W,e]}=[\overline c_{V,W},e]$ for some natural isomorphism $\overline c_{V,W}:V\otimes  W\to W\otimes  V$, for all $V,W$ objects in $\mathcal C$.  Since the associativity isomorphism satisfies $a_{[V,e],[W,e],[U,e]}=[\overline a_{V,W,U},e]$, where $\overline a$ is the associativity morphism for $\mathcal C$; then $\overline c$ is a braiding for $\mathcal C$.

\end{proof}

\bigbreak

In \cite{G}, the author describes and classifies a family of such extensions and calls it crossed product tensor category. Fix $H$ a finite-dimensional Hopf algebra. In the case when $\mathcal C=\operatorname{Comod}(H)$, in \cite{MM}, we described crossed products in terms of Hopf-algebraic datum. A continuation they are introduced. 

If $g\in G(H)$ and $L$ is a
$(H,H)$-biGalois object then the cotensor product $L\Box_H \Bbbk_g$
is one-dimensional. Let $\phi(L,g)\in \Gamma$ be the group-like element such that
$L\Box_H \Bbbk_g\simeq \Bbbk_{\phi(L,g)}$ as left $H$-comodules.
Assume that $A$ is an $H$-biGalois object with
left $H$-comodule structure $\lambda:A\to H\otimes_\Bbbk A$. If $g\in G(H)$ is a group-like
element we can define a new  $H$-biGalois object  $A^g$ on
the same underlying algebra $A$ with unchanged right comodule structure
and a new left $H$-comodule structure given by  $\lambda^g: A^g \to H\otimes_\Bbbk A^g$,
$\lambda^g(a)=g^{-1}a_{-1}g\otimes  a_0$
for all $a\in A$.

\begin{thm} \label{crossed}\cite[Lem 5.7, Thm 5.4]{MM} Let $\Upsilon=(L_a,(g(a,b),f^{a,b}),\gamma)_{a,b\in \Gamma}$ be a collection where
\begin{itemize}
	\item  $L_a$ is a $(H,H)$-biGalois object;
	\item  $g(a,b)\in G(H)$;
	\item  $f^{a,b}:(L_a\Box_H L_{b})^{g(a,b)}\to L_{ab}$ are bicomodule algebra isomorphisms;
	\item $\gamma\in Z^3(G(H),\Bbbk^\times)$ normalized,
\end{itemize}
such that for all $a,b,c\in \Gamma$:
\begin{equation}\label{cross-hopf-11}
L_e=H, \quad (g(e,a),f^{e,a})=(e,\operatorname{id}_{L_a})= (g(a,e),f^{a,e});
\end{equation}
\begin{equation}\label{cross-hopf-12}\phi(L_a,g(b,c)) g(a,bc)= g(a,b) g(ab,c);
\end{equation}
\begin{equation}\label{cross-hopf-14}
f^{ab,c}(f^{a,b}\otimes\operatorname{id}_{L_c})=f^{a,bc}(\operatorname{id}_{L_a}\otimes  f^{b,c}).
\end{equation}
Then $\operatorname{Comod}(H)(\Upsilon):=\oplus_{g\in\Gamma}\operatorname{Comod}(H)$ as a structure of tensor category.
\end{thm}

\begin{proof} We give an sketch of the proof.  Let $\Upsilon$ be a collection as in the Theorem. For $V,W\in\operatorname{Comod}(H)$, $a,b\in\Gamma$, define
\begin{align*}
    [V,a]\otimes  [W,b]:&=[V\otimes  (L_a\Box_H W)\otimes  \Bbbk_{g(a,b)},ab],\\
    [V,1]^*:&=[V^*,1],\\
    [\Bbbk,a]^*:&=[\Bbbk_{g(a,a^{-1})},a^{-1}].
\end{align*} Using \cite[Eq (5.8)]{MM}, we obtain the pentagon diagram and therefore $\operatorname{Comod}(H)(\Upsilon)$ is a monoidal category. Since $\operatorname{Comod}(H)$ is finite tensor category, then $\operatorname{Comod}(H)(\Upsilon)$ is also finite tensor category.
\end{proof}{}

The following Lemma gives us a second condition to decided if our  extensions can be braided.
\begin{thm}\label{braided}
If $\operatorname{Comod}(H)(\Upsilon)$ is braided with braiding $c$ then the following conditions have to hold
\begin{enumerate}
    \item $L_a\simeq H$ for all $a\in \Gamma$, 
    \item $\Gamma$ is Abelian,
    \item $\Upsilon$ comes from a data $(g,f^{a,b},\gamma)_{a,b\in\Gamma}$ with 
\begin{itemize}
    \item $g\in Z^2(\Gamma,G(H))$ normalized,
    \item $f^{a,b}:H^{g(a,b)}\to H$ a bicomodule algebra isomorphism with $f^{ab,c}f^{a,b}=f^{a,bc}f^{b,c}$,
    \item $\gamma\in Z^3(G(H),\Bbbk^{\times})$ normalized.
\end{itemize} \end{enumerate}
\end{thm}

\begin{proof} (1)  
Take, for any $V\in\operatorname{Comod}(H)$, $c_{[V,e][\mathbf 1,a]}:[V,a]\to[L_a\Box_H V,a]$, 
this defines a natural isomorphism $\overline c_a:\operatorname{id}_{\mathcal C}\to L_a\Box_H -$ wich is monoidal since $c$ is a braiding. Then $L_a\simeq H$ as bicomodule algebras for all $a\in\Gamma$.

(2) Consider $c_{[\mathbf 1,a][\mathbf 1,b]}:[\Bbbk_{g(a,b)},ab]\to[\Bbbk_{g(b,a)},ba]$ then $ab=ba$ for all $a,b\in\Gamma$ and $\Gamma$ is Abelian.

(3) Since $L_a$ is trivial, then condition Theorem \ref{crossed}(\ref{cross-hopf-12}) is equivalent to  $g\in Z^2(\Gamma,G(H))$ and it is normalized by condition Theorem \ref{crossed}(\ref{cross-hopf-11}). Moreover $f^{a,b}:H^{g(a,b)}\to H$ is a bicomodule algebra isomorphism that satisfies $f^{ab,c}f^{a,b}=f^{a,bc}f^{b,c}$ which is equivalent to condition Theorem \ref{crossed}(\ref{cross-hopf-14}).

\end{proof}
\begin{rem} By definition of bicomodule morphism, $f^{a,b}:H\to H$ has to be an algebra isomorphism such that $f^{a,b}(h)_1\otimes  f^{a,b}(h)_2=g^{-1}h_1g\otimes  f^{a,b}(h_2)$ and $f^{a,b}(h)_1\otimes  f^{a,b}(h)_2=f^{a,b}(h_1)\otimes  h_2$, then $g^{-1}h_1g\otimes  f^{a,b}(h_2)=f^{a,b}(h_1)\otimes  h_2$.

\end{rem}

In the case when $H=\wedge V\# \Bbbk C_2$, as Example \ref{super}, using the previous Theorem we obtained eight tensor categories non-equivalent pairwise, \cite[Sect 6.3]{MM}, named $\mathcal C_0(1,\operatorname{id},\pm 1)$, $\mathcal C_0(u,\iota,\pm 1)$, $\mathcal D(1,\operatorname{id},\pm 1)$, $\mathcal D(u,\iota,\pm 1)$.

In all cases, the underlying Abelian category is $\operatorname{Comod}(H) \oplus
\operatorname{Comod}(H)$ and for $V,W,Z\in \operatorname{Comod}(H)$ they are defined in the following way:

$\bullet$ The tensor product, dual objects and associativity in  $\mathcal C_0(1,\operatorname{id},\pm 1)$ are given by
\begin{align*}
[V,e][W,g]&=[V\otimes  W,g] ,& [V,u][W,g]&=[V\otimes  \textbf{U}_0\Box_H W,ug],\\
[V,e]^*&=[V^*,e],& [\mathbf{1},u]^*&=[\Bbbk,u],\end{align*}
$\alpha_{[V,u],[W,u],[Z,u]}$ is not trivial, 
and $\textbf{U}_0$ is certain BiGalois object, see \cite[Section 4]{MM}.

$\bullet$ The tensor product, dual objects and associativity in $\mathcal C_0(u,\iota,\pm 1)$ are given by
\begin{align*}
[V,e][W,e]&=[V\otimes  W,1] ,& [V,u][W,u]&=[V\otimes  \textbf{U}_0\Box_H W\otimes\Bbbk_u,e],\\
[V,e][W,u]&=[V\otimes  W,u] ,& [V,u][W,e]&=[V\otimes  \textbf{U}_0\Box_H W,u],\\
[V,e]^*&=[V^*,e],& [\mathbf{1},u]^*&=[\Bbbk_u,u],\end{align*}
$\alpha_{[V,u],[W,u],[Z,u]}$ is not trivial.

$\bullet$ The tensor product, dual objects and associativity in $\mathcal D(1,\operatorname{id},\pm 1)$ are given by
\begin{align*}
[V,e][W,g]&=[V\otimes  W,g] ,& [V,u][W,g]&=[V\otimes  W,ug],\\
[V,e]^*&=[V^*,e],& [\mathbf{1},u]^*&=[\Bbbk,u],\end{align*}
$\alpha_{[V,u],[W,u],[Z,u]}=[\pm \operatorname{id}_{V\otimes  W\otimes  Z},u]$ and the others are trivial.

$\bullet$ The tensor product, dual objects and associativity in $\mathcal D(u,\iota,\pm 1)$ are given by
\begin{align*}
[V,e][W,e]&=[V\otimes  W,e] ,& [V,u][W,u]&=[V\otimes  W\otimes\Bbbk_u,e],\\
[V,e][W,u]&=[V\otimes  W,u] ,& [V,u][W,e]&=[V\otimes  W,u],\\
[V,e]^*&=[V^*,e],& [\mathbf{1},u]^*&=[\Bbbk_u,u],\end{align*}
$\alpha_{[V,u],[W,u],[Z,u]}=[\pm\operatorname{id}_{V\otimes  W} \otimes\tau(\varepsilon\iota\rho_Z\otimes\operatorname{id}_{Z\otimes  \Bbbk_u}),u],$ where $\iota:H^u\to H$ is the unique bicomodule algebra isomorphism which satisfies $\iota(u)=-u$ and $\iota(x)=-x$ for $x\in V$; and $\tau:X\otimes  Y\to Y\otimes  X$, $\tau(z\otimes  k)=k\otimes  z$ for all $X,Y\in \operatorname{Comod}(H)$, see \cite[Rem 2.2]{MM}.

\begin{rem} By Lemma \ref{braided}(1), we obtain that only the categories $\mathcal D(1,\operatorname{id},\pm 1)$ and $\mathcal D(u,\iota,\pm 1)$ could be braided, since the BiGalois objects have to be trivials.

By direct calculation on Equation (\ref{bra1}), $\mathcal D(u,\iota,- 1)$ is not braided with trivial braiding. So, in this case, we want to know if there exist another possible braidings. 
\end{rem}

\section{Braided crossed product}
Let $\Gamma$ be an Abelian group. In \cite{MM}, following the ideas developed in \cite{G}, we described all $\Gamma$-crossed product tensor categories which are extensions of $\operatorname{Comod}(H)$ for $H$ a Hopf algebra in terms of certain Hopf-algebraic datum. Fix $(H,r)$ a  CQT-Hopf algebra. In the first Lemma of this Section, we do the same for the braiding of crossed products that are $\Gamma$-extensions of $\operatorname{Comod}(H)$.

\begin{rem} If $v:H\to H$ is a left $H$-comodule morphism, since the coaction is the coproduct, $v$ satisfies $v(x)_1\otimes  v(x)_2=x_1\otimes  v(x_2)$, for all $x\in H$. In particular, $v$ is not a coalgebra morphism and if $g\in G(H)$, $v(g)=g\varepsilon(v(g))$. 

\end{rem}

\begin{lem} Fix a datum $(g,f^{a,b},\gamma)_{a,b\in\Gamma}$, as in Lemma \ref{braided}, and let $\mathcal C$ be the associated tensor category. Consider a pair $(v^a,w^a)_{a\in\Gamma}$ where $v^a,w^a:H\to H$ are left $H$-comodule algebra isomorphisms. 
Let $W^a=\varepsilon w^a$, $V^a=\varepsilon v^a$ and $F^{a,b}=\varepsilon f^{a,b}$.
If for all $a,b,c\in\Gamma$ and $X\in\operatorname{Comod}(H)$ we have 
\begin{equation}
    v^1 =w^1= \operatorname{id}_H,
\end{equation}
\begin{equation}\label{braid1CCOA}
(g(a,b),f^{a,b})=(g(b,a),f^{b,a}), \end{equation}
\begin{equation}\label{braid2CCOA}
W^b(x_{-3})W^a(x_{-2})(W^{ab})^{-1}(x_{-1})x_0=F^{a,b}(x_{-2})r(x_{-1}\otimes  g(a,b))x_0,\quad x\in X, \end{equation}
\begin{equation}\label{braid3CCOA}
V^b(x_{-3})V^a(x_{-2})(V^{ab})^{-1}(x_{-1})x_0=r(x_{-2}\otimes  g(a,b))F^{a,b}(x_{-1})x_0,\quad x\in X,\end{equation}
\begin{equation}\label{braid4CCOA}
V^a(g(b,c))=(\gamma_{a,b,c} \gamma_{b,c,a})^{-1}\gamma_{b,a,c}, \end{equation}
\begin{equation}\label{braid5CCOA}
W^b(g(c,a))=\gamma_{c,a,b} \gamma_{b,c,a}\gamma_{c,b,a}^{-1}; \end{equation}

then we obtain a braiding over $\mathcal C$ given by
 $$\mathbf c_{[V,a],[W,b]}=c_{V,W}((V^a\otimes\operatorname{id})\rho_V\otimes  (W^a\otimes\operatorname{id})\rho_W)\otimes\operatorname{id}, \ \ V,W\in\operatorname{Comod}(H), a,b\in\Gamma.$$
All braidings over $\mathcal C$ come from a pair $(v^a,w^a)_{a\in\Gamma}$ which satisfies (\ref{braid1CCOA}) to (\ref{braid5CCOA}).
\end{lem}

\begin{proof} By \cite[Defn 5.3]{G}, a datum $(g,f^{a,b},\gamma)_{a,b\in\Gamma}$ has associated a braiding if there exist a triple $(\theta^a, \tau^a, t_{a,b})_{a,b\in G}$ where
\begin{itemize}
  \item $\theta^a, \tau^a:\operatorname{id}_{\mathcal C}\to \operatorname{id}_{\mathcal C}$ are  monoidal natural isomorphisms,
  \item for all $a,b \in G$,  $t_{a,b}:(U_{a,b},\sigma^{a,b})\to(U_{b,a},\sigma^{b,a}) $ are isomorphisms in $\mathcal{Z}(\mathcal C)$, where $\sigma^{a,b}_X=\tau(\varepsilon f^{a,b}\otimes\operatorname{id}_X)\rho_X$, for $(X,\rho_X)\in \operatorname{Comod}(H)$, and $U_{a,b}=\Bbbk_{g(a,b)}$,
\end{itemize}
such that for all $a,b,c\in\Gamma$ and $X\in\mathcal C$, the following conditions hold
\begin{equation}\label{br1}\theta^1 =\tau^1= \operatorname{id},\quad \theta_\mathbf 1^a =\operatorname{id}_\mathbf 1=\tau_\mathbf 1^a,\quad t_{a,1}=t_{1,a}=\operatorname{id}_\mathbf 1,\end{equation} 
\begin{equation}\label{br2}c_{U_{a,b},X}\sigma^{a,b}_X=((\tau^{ab}_X)^{-1}\tau^{a}_X\tau^b_X)\otimes\operatorname{id}_{U_{a,b}},\end{equation}
\begin{equation}\label{br3}\sigma^{a,b}_Xc_{U_{a,b},X}
=\operatorname{id}_{U_{a,b}}\otimes((\theta^{ab}_X)^{-1}\theta^a_X\theta^b_X),\end{equation}
\begin{equation}\label{br4}\gamma_{a,b,c}(\theta^a_{U_{b,c}}\otimes  t_{bc,a})\gamma_{b,c,a}=(t_{b,c}\otimes\operatorname{id}_{U_{ba,c}})\gamma_{b,a,c}(t_{c,a}\otimes\operatorname{id}_{U_{b,ac}}),\end{equation}
\begin{equation}\label{br5}\gamma^{-1}_{c,a,b}(\tau^b_{c,a}\otimes  t_{b,ca})\gamma^{-1}_{b,c,a}=(t_{b,a}\otimes\operatorname{id}_{U_{c,ba}})\gamma^{-1}_{c,b,a}(t_{b,c}\otimes\operatorname{id}_{U_{bc,a}}).\end{equation}

By Lemma \ref{Sch}, each monoidal natural isomorphism of the identity functor comes from a left $H$-comodule algebra isomorphism, then $\theta^a_X:=(\varepsilon v^a \otimes\operatorname{id})\rho_X$ and $\tau^a_X:=(\varepsilon w^a\otimes  \operatorname{id}_X)\rho_X$ for all $X\in\operatorname{Comod}(H)$. Since $U_{g(a,b)}=\Bbbk_{g(a,b)}$, we can take $t_{a,b}\in\Bbbk^*$.

Each $t_{a,b}$ is a left $H$-comodule isomorphism if and only if $g(a,b)\otimes  t_{a,b}\operatorname{id}_{\Bbbk}=g(b,a)\otimes  t_{a,b}\operatorname{id}_\Bbbk$  wich gives $g(a,b)=g(b,a)$ for all $a,b\in\Gamma$. Moreover, each $t_{a,b}$ is a braided morphism if and only if $\sigma^{a,b}_Xt_{a,b}=\sigma^{b,a}_X t_{a,b}$ for $a,b\in\Gamma$ and $X\in\operatorname{Comod}(H)$ if and only if   $\sigma^{a,b}=\sigma^{b,a}$. Then $t_{a,b}$ is an isomorphism in $\mathcal Z(\operatorname{Comod}(H))$ if and only if Condition (\ref{braid1CCOA}) holds.

Condition (\ref{br1}) is equivalent to  $v^1 =w^1= \operatorname{id}_H$ and $t_{a,1}=t_{1,a}=1$, since $\theta_\Bbbk^a =\operatorname{id}_\Bbbk=\tau_\Bbbk^a$ is always true.
Condition (\ref{br2}) is equivalent to $$F^{a,b}(x_{-2})r(x_{-1}\otimes  g(a,b))x_0\otimes  k=W^b(x_{-3})W^a(x_{-2})(W^{ab})^{-1}(x_{-1})x_0\otimes  k,$$ 
for $x\otimes  k\in X\otimes\Bbbk_{g(a,b)}$, which is equivalent to Condition (\ref{braid2CCOA}). In the same way, Condition (\ref{br3}) is equivalent to Condition (\ref{braid3CCOA}).
Condition (\ref{br4}) is equivalent to $\gamma_{a,b,c}V^a(g(b,c)) t_{bc,a}\gamma_{b,c,a}=t_{b,c}\gamma_{b,a,c}t_{c,a}$ but if we take $c=1$ then $$1=t_{b,a}, \text{ for }a,b\in\Gamma,$$ so, this Condition is equivalent to Condition (\ref{braid4CCOA}), and Condition (\ref{br5}) is equivalent to Condition (\ref{braid5CCOA}).

By \cite[Thm 5.4]{G}, this pair produces a braiding over $\mathcal C$ given by
\begin{align}\label{trenza}\mathbf c_{[V,a],[W,b]}=[c_{V,W}(\theta^a_V\otimes\tau^a_W ),ab], \ \ \text{ for all }V,W\in\operatorname{Comod}(H), a,b\in\Gamma,\end{align} and all braidings come from such a pair.
\end{proof}

Now, we focus our attention into the case $\Gamma=C_2$. By Lemma \ref{braided}, a datum $\Upsilon'=(g,f,\gamma)$ with $g\in G(H)$ a group-like element, $f:H^g\to H$ a bicomodule algebra isomorphism and $\gamma\in\Bbbk^{\times}$, $\gamma^2=1;$ generates a tensor category $\mathcal C=\operatorname{Comod}(H)(\Upsilon')$.

\medbreak

The following Theorem gives us the third and last condition to decide if our categories are braidable.
\begin{thm}
The category $\operatorname{Comod}(H)(\Upsilon')$ is a braided $C_2$-extension if and only if, there exists
a pair of isomorphisms of left $H$-comodule algebras $v,w:H\to H$ such that for all $X\in\operatorname{Comod}{H}$ and $x\in X$ 
\begin{enumerate}
    \item[a.] $\varepsilon(w(x_{-2})w^{-1}(x_{-1}))x_0=x$,
    \item[b.] $\varepsilon(w(x_{-2})w(x_{-1}))x_0=\varepsilon f(x_{-2})r(x_{-1}\otimes  g)x_0$,
    \item[c.] $\varepsilon(v(x_{-2})v^{-1}(x_{-1}))x_0=x$,
    \item[d.] $\varepsilon(v(x_{-2})v(x_{-1}))x_0=r(x_{-2}\otimes  g)\varepsilon f(x_{-1})x_0$,
    \item[e.] $\varepsilon(v(g))=\gamma^{-1}$,
    \item[f.] $\varepsilon(w(g))=\gamma$,
\end{enumerate}
\end{thm}
\begin{proof}

Condition (\ref{braid1CCOA}) is always true. Condition (\ref{braid2CCOA}) is equivalent to $r(x_{-1}\otimes  1)x_0=x$, and items a,b. Condition (\ref{braid3CCOA}) is equivalent to  $r(x_{-1}\otimes  1)x_0=x$, and items c,d. Condition (\ref{braid4CCOA}) is equivalent to item e. Condition (\ref{braid5CCOA}) is equivalent to item f.

Regarding condition $r(x_{-1}\otimes  1)x_0=x$, it is always true over a CoQuasi-triangular Hopf algebra.

\end{proof}

If $H=\wedge V\# \Bbbk C_2$, as Example \ref{super}, by \cite[Prop 4.10]{MM}, the isomorphisms $v$ and $w$ are identities. Then if the extension is braided the only possible braiding is the trivial, see Equation (\ref{trenza}), since the category $\operatorname{Comod}{H}$ has a braiding giving by the r-form. With this information, Conditions a-f are equivalent to 
\begin{enumerate}
    \item[a'.]  $\varepsilon(x_{-2}x_{-1})x_0=x$,
    \item[b'.]$\varepsilon(x_{-2}x_{-1})x_0=\varepsilon f(x_{-2})r(x_{-1}\otimes  g)x_0$,
    \item[c'.] $\varepsilon(x_{-2}x_{-1})x_0=r(x_{-2}\otimes  g)\varepsilon f(x_{-1})x_0$,
    \item[d'.] $\varepsilon(g)=\gamma^{-1}$,
    \item[e'.] $\varepsilon(g)=\gamma$.
\end{enumerate}
Since $g$ is a group-like element, d' and e' imply that $\gamma=1$. Thus, the only categories that could be braided are $\mathcal D(1,\operatorname{id},1)$ and $\mathcal D(u,\iota, 1)$.

\begin{cor} A $C_2$-extension over $\operatorname{Comod}{(\wedge V\# \Bbbk C_2)}$ is braided if and only if, for all comodule $X$, 
$r(f(x_{-1})\otimes  g)x_0=x, \text{ for all } x\in X.$
 \end{cor}
 
\begin{proof}  Condition a' is always true over comodules.
Since $x_1y_1r(x_2\otimes  y_2)=r(x_1\otimes  y_1)y_2x_2$ for $x,y\in H$ we have
\begin{align*}
    (x_{-1}g)r(x_{-2}\otimes  g)\otimes  x_0=r(x_{-1}\otimes  g)gx_{-2}\otimes  x_0.
\end{align*}
Applying $\varepsilon f\otimes\operatorname{id}_X$, we obtain $r(x_{-2}\otimes  g)\varepsilon f(x_{-1})x_0=\varepsilon f(x_{-2})r(x_{-1}\otimes  g)x_0$. This implies that Conditions b' and c' are equivalent.
Since $$r(f(x)\otimes  g)=r(f(x)_1\otimes  g)\varepsilon(g(f(x)_2))=r(x_1\otimes  g)\varepsilon(f(x_2))$$
we have $r(f(x_{-1})\otimes  g)x_0=\varepsilon f(x_{-2})r(x_{-1}\otimes  g)x_0$, then Condition b' is equivalent to $$r(f(x_{-1})\otimes  g)x_0=x.$$
\end{proof}
We are ready for our main result.
\begin{thm}
The categories $\mathcal D(1,\operatorname{id},1)$ and $\mathcal D(u,\iota, 1)$ are braided tensor categories. The remaining 6 categories found in \cite{MM}  are non-braidable.
\end{thm}

\begin{proof}
By \cite[Thm 5.4]{G}, the only possible option for $v$ and $w$ is for there to be the identity. 
Then the categories $\mathcal D(1,\operatorname{id},1)$ and $\mathcal D(u,\iota, 1)$
have associated at most a single pair $(\operatorname{id},\operatorname{id})$, which would give it a braided structure For the remaining six categories, we already know that they are non-braidable. 

Since $\mathcal D(1,\operatorname{id},1)$ has trivial associativity and $\operatorname{Comod}(H)$ is braided then the braiding for $\mathcal D(1,\operatorname{id},1)$ is  \begin{align}\label{aa}\mathbf c_{[V,a],[W,b]}=[c_{V,W},ab], \ \ \text{ for all }V,W\in\operatorname{Comod}(H), a,b \in C_2.\end{align} 

Over $\mathcal D(u,\iota, 1)$ it is enough to check Equations (\ref{bra1}) and (\ref{bra2}) where the associativity is not trivial. Since $(f\otimes  \operatorname{id})(\operatorname{id}\otimes  g)=(\operatorname{id}\otimes  g)(f\otimes  \operatorname{id})$ for any $f,g$ morphisms in the category, also the braiding given in \eqref{aa} also satisfies the desired Equations. 
\end{proof}

\begin{cor} For $X\in \operatorname{Comod}{(\wedge V\# \Bbbk C_2)}$,  $r(\iota(x_{-1})\otimes  g)x_0=x$, for all $ x\in X$.
 \end{cor}

\begin{rem}
Since $\operatorname{Comod}(H)$ is not symmetric, then these two categories are not symmetric either. 
\end{rem}


\begin{thebibliography}{99}

\bibitem{AG} N. Andruskiewitsch and M. Grana, {\em Braided Hopf Algebras over Non Abelian Finite Groups}, Conference Proceedings Acad. Nac. Ciencias (1997).

\bibitem{EG} P. Etingof and S. Gelaki, {\em The classification of finite-dimensional triangular Hopf algebras over an algebraically closed field of characteristic 0},
Moscow Math J., 3(1) (2003), 37--43.


\bibitem{G}  C. Galindo, {\em Crossed product tensor categories}, J.  Algebra, 337  (2011), 233--252.

\bibitem{HH} T. Hagge and S.-M. Hong, {\em Some non-braided fusion categories of rank 3}, Commun. Contemp. Math., 11(4) (2009), 615--637. 


\bibitem{JS} A. Joyal and R. Street, {\em Braided monoidal categories}, Macquarie Mathematics Reports (1986).


\bibitem{Ma1} A. Masuoka, {\em Cocycle deformations and Galois objects for some cosemisimple Hopf algebras of finite dimension}, in “New trends in Hopf algebra theory (La Falda, 1999)”, Contemp. Math., 267 (2000), 195-–214. 

\bibitem{Ma2} A. Masuoka, {\em Example of almost commutative Hopf algebras which are not coquasitriangular Hopf algebras}, Lecture Notes in Pure and Appl. Math., 237 (2004), 185-191.


\bibitem{MM} A. Mejia and M. Mombelli, {\em Crossed extensions of the corepresentation category of finite supergroup algebras}. Int. J. Math. DOI :  10.1142/S0129167X15500676.

\bibitem{M} S. Montgomery, {\em Hopf algebras and their actions on rings}, CBMS Regional Conf. Ser. in Math., American Mathematical Society, 82 (1993), 194--195.

\bibitem{S} P. Schauenburg, {\em Hopf Bigalois extensions}, Comm. in Algebra, 24 (1996), 3797--3825.

\end{thebibliography}
\end{document}